\documentclass[11pt]{amsart}
\usepackage{amsbsy,amssymb,amscd,amsfonts,latexsym,amstext,delarray, amsmath,graphicx,color,caption, amsthm, enumerate}
\usepackage{graphicx}
\usepackage{mathtools}
\usepackage{hyperref}
\usepackage{epsfig}
\input xy

\def\Z{{\mathbb Z}}
\def\R{{\mathbb R}}

\def\E{{\mathbb E}}

\hypersetup{
	colorlinks=true,
	linkcolor=blue,
	filecolor=blue,      
	urlcolor=blue,
	citecolor = blue
}

\begin{document}

\newtheorem{definition}{Definition}[section]
\newtheorem{example}{Example}[section]
\newtheorem{lemma}{Lemma}[section]
\newtheorem{thm}{Theorem}
\newtheorem{prop}[lemma]{Proposition}
\newtheorem{cor}[lemma]{Corollary}

\title[Sparse Domination Using Square Functions]{Sparse Domination for Bi-Parameter Operators using Square functions}
\author{Alexander Barron and Jill Pipher}

\maketitle

\newcommand{\Addresses}{{
		\bigskip
		\footnotesize

		\textsc{Department of Mathematics, Brown University,
			Providence, RI 02906, USA}\par\nopagebreak
		\textit{E-mail addresses}: \texttt{alexander\_barron@brown.edu,  jpipher@math.brown.edu}

	}}

\begin{abstract}
	Let $S$ be the dyadic bi-parameter square function $$Sf(x)^{2} = \sum_{R \in \mathcal{D}} |\langle f, h_{R} \rangle|^{2} \frac{1_{R}(x)}{|R|}.$$ We prove that if $T$ is a bi-parameter martingale transform and $f,g$ are suitable test functions, then there exists a sparse collection of rectangles $\mathcal{S}$ such that $$|\langle Tf, g \rangle| \lesssim \sum_{R \in \mathcal{S}} |R|(Sf)_{R}(Sg)_{R}.$$ We also extend this estimate to the case where $T$ is a bi-parameter cancellative dyadic shift and when $T$ is a paraproduct-free singular integral of Journ\'{e} type. Weighted estimates follow from the domination.  
\end{abstract}

\section{Introduction}

The theory of sparse domination is a recent addition to the classical theory of singular integral operators. One begins with an operator $T$, for example a Calder\'{o}n-Zygmund operator, and then shows that for suitable test functions $f$ the estimate 
\begin{equation} \label{classicSparse}
|Tf| \lesssim \sum_{Q \in \mathcal{S}} (|f|)_{Q}1_{Q}
\end{equation} holds in some sense. 
Here $(|f|)_{Q}$ denotes the average of $|f|$ over the cube $Q$, and
$\mathcal{S}$ is a \textit{sparse} collection of cubes in $\R^{n}$, meaning that there is some $\eta > 0$ such that for every $Q \in \mathcal{S}$ we can find $E_{Q} \subset Q$ with $|E_{Q}| > \eta |Q|$ and moreover $E_{Q} \cap E_{Q'} = \emptyset$ for $Q \neq Q'$.
The ``sense" in which this domination holds ranges from norm bounds, as in Lerner's original paper in the subject \cite{Lerner0}, 
to pointwise bounds.
Here the collection $\mathcal{S}$ depends on the function $f$, but the sparse bound can be used as an intermediate step to prove other estimates of interest. For example, it is straightforward to recover Hyt\"{o}nen's sharp $A_{2}$ bound \cite{Hy} for Calder\'{o}n-Zygmund operators by using \eqref{classicSparse}, and indeed sparse bounds yield sharp weighted estimates for a variety of operators.

In the case where $T$ is a Calder\'{o}n-Zygmund operator or a dyadic shift operator, pointwise bounds were proven in \cite{CondeRey} and \cite{DyCalc}, and later in and \cite{lacey1} and \cite{Lerner2}. One can also show that \eqref{classicSparse} holds in the \textit{sparse form} sense, meaning that for suitable test functions $f,g$ there exists a sparse collection $\mathcal{S}$ such that \begin{equation}\label{stdSparseForm}|\langle Tf, g \rangle| \lesssim \sum_{Q \in \mathcal{S}} |Q| (|f|)_{Q}(|g|)_{Q}.\end{equation} See \cite{roughSparse} for a proof.  While a pointwise estimate of the type \eqref{classicSparse} is stronger than the form bound \eqref{stdSparseForm},  in many applications \eqref{stdSparseForm} is sufficient, for example in proving sharp weighted estimates. Moreover, the sparse form technique has led 
to sparse bounds for several operators of interest that fall outside the scope of the classical Calder\'{o}n-Zygmund theory. Examples include rough singular integrals (linear \cite{roughSparse} and bilinear \cite{B}), bilinear Hilbert transforms \cite{MultiSparse}, discrete singular integrals \cite{KrauseLacey}, the spherical maximal function \cite{lacey2}, and various singular operators associated to semigroups \cite{BFP}.  

It is natural to ask whether or not there is some analogue of a sparse bound for bi-paramater singular integrals. Suppose, for example, that $$Tf(x_1,x_2) = \text{p.v.} \int_{\R^{2}} \frac{f(y_1, y_2)}{(x_1 - y_1) (x_2 - y_2)} dy_1 dy_2,$$ so that $T = H_1 \otimes H_2$ with $H_{1}$ a Hilbert transform in the $x_1$ direction and $H_2$ a Hilbert transform in the $x_2$ direction. If we fix the variable $x_1$ then we can apply the one-parameter estimate \eqref{classicSparse} to show that there is a sparse collection of intervals $\mathcal{S}_{x_1}$ such that $$|Tf(x_1, x_2)| \lesssim  \sum_{Q \in \mathcal{S}_{x_1}} (|H_{1} f(x_1, \cdot)|)_{Q}1_{Q}(x_2) $$ for almost every $x_2$. However, since the collection of intervals depends on $x_1$ there is no obvious way to iterate this estimate to get a sparse bound for the full operator $H_1 \otimes H_2$. We encounter a similar issue when trying to iterate the sparse form bound \eqref{stdSparseForm}. We therefore need to find a more direct approach that does not rely on the one-parameter results. It is clear that any analogue of the one-parameter sparse bound must involve collections of rectangles rather than cubes, due to the underlying geometry of bi-parameter singular integrals like $H_1 \otimes H_2$. At this point one encounters substantial difficulties adapting the one-parameter methods. For example, to prove \eqref{classicSparse} and \eqref{stdSparseForm} we construct the sparse collection of cubes using a stopping-time argument that is intimately related to the Calder\'{o}n-Zygmund decomposition and the Hardy-Littlewood maximal operator. In the bi-parameter setting the natural maximal operator to work with is the \textit{strong maximal function} $$M_{S}f(x) = \sup_{x \in R} \frac{1}{|R|} \int_{R} |f(y)| dy,$$ where the supremum is taken over rectangles containing $x$. However, this operator lacks the martingale structure that enables the type of stopping-time arguments used in the one-parameter setting. In the bi-parameter setting, for both singular integral theory and martingale theory, the square function is the most natural operator. For example, in \cite{Bernard} the square function was used to give an atomic decomposition for the Hardy space and prove $H^1 - BMO$ duality, later extended to the continuous setting of bi-parameter singular integrals, Hardy spaces, and product $BMO$ in \cite{ChF2}. 

In the bi-parameter setting, it is also not immediately clear what the proper definition of a `sparse collection of rectangles' should be. There are two likely candidates: 

\begin{definition}\label{sparseDef1} A collection $\mathcal{S}$ of rectangles in $\R^{n}$ is said to be \textit{sparse in the disjoint-pieces sense} if there is some $\eta> 0$ such that for all $R \in \mathcal{S}$ there is $E_R \subset R$ with $|E_{R}| > \eta |R|$, and such that if $R \neq R'$ then $E_{R} \cap E_{R'} = \emptyset$.  
\end{definition}

\begin{definition}\label{sparseDef2} A collection $\mathcal{S}$ of rectangles in $\R^{n}$ satisfies the \textit{Carleson packing condition} if there is some $\Lambda > 0$ such that for all open sets $U \subset \R^{n}$, $$\sum_{\substack{R \in \mathcal{S} \\ R \subset U}} |R| \leq \Lambda|U|.$$ \end{definition}

\noindent The structure of the {\it packing condition} in definition \ref{sparseDef2} is natural in light of the fact that the definition of bi-parameter $BMO$ requires 
a similar packing condition on Haar or wavelet coefficients relative to rectangles contained in open sets.
Both definitions are equivalent for collections of cubes \cite{DyCalc}, but as far as we know this equivalence is an open problem for rectangles. It is clear that the disjoint-pieces condition implies the Carleson packing condition, but we do not know if the reverse implication is true or false. 

The goal of this paper is to provide one approach to a sparse bound for certain bi-paramter operators, including the generalizations of $H_1 \otimes H_2$
beyond the tensor product, or even convolution, structure. Let $S$ be the dyadic bi-parameter square function given by $$Sf(x)^{2} = \sum_{R \in \mathcal{D}} |\langle f, h_{R} \rangle|^{2} \frac{1_{R}(x)}{|R|},$$ where $\mathcal{D} = \mathcal{D}_1 \times \mathcal{D}_2$ is the collection of dyadic rectangles in $\R^{2}$ (relative to two grids $\mathcal{D}_1$ and $\mathcal{D}_2$ in $\R$), and $h_{R}$ is the bi-parameter Haar function associated to $R$. Recall that if $R = I \times J$ then $$h_{R}(x_1,x_2) = h_{I}(x_1)h_{J}(x_2), \ \ \ \ \ \text{  } h_{I} = \frac{1}{|I|^{1/2}}(1_{I_{l}} - 1_{I_{r}}), $$ where $I_{l}$ and $I_{r}$ are the left and right children of $I$. Also recall that the functions $h_{R}$ form a basis of $L^{2}(\R)$ for any $\mathcal{D}$. Given the role of square functions in the bi-parameter theory, it is natural to attempt to prove a sparse form bound of the type \begin{equation} \label{squareSparse} |\langle Tf, g \rangle| \lesssim \sum_{R \in \mathcal{S}} |R|(Sf)_{R}(Sg)_{R},\end{equation} where $\mathcal{S}$ is a collection of rectangles that satisfies either the disjoint-pieces or Carleson packing condition. 

We begin by studying the \textit{bi-paramater martingale transform} $$Tf = \sum_{R \in \mathcal{D}}\epsilon_R\langle f, h_R \rangle h_R, \ \ \ \ \ \sup_{R}|\epsilon_{R}| \leq C$$ and prove that in this case the square-function sparse bound \eqref{squareSparse} holds.

\begin{thm}\label{mainThm} Let $T$ be the bi-parameter martingale transform defined above, and suppose $f$ and $g$ are functions with finitely many Haar coefficients. Then there exists a collection of rectangles $\mathcal{S}$ that is sparse in the disjoint-pieces sense (Definition \ref{sparseDef1}) such that 
	\begin{equation}\label{SSparse}
	|\langle Tf, g \rangle| \lesssim (\sup_{R} |\epsilon_R|) \sum_{R \in \mathcal{S}} |R|(Sf)_{R}(Sg)_{R}.
	\end{equation}
	The implicit constant does not depend on $f$ or $g$. \end{thm}  

The methods used to prove Theorem \ref{mainThm} generalize to the case where $T$ is a bi-parameter dyadic shift as long as we replace the dyadic square functions $S$ by certain shifted square functions $S^{i,j}$. Since the statement of this result is somewhat technical, we defer the detailed definitions until Section 3. As a consequence of Martikainen's representation theorem \cite{Mart}, we are then able to deduce a type of sparse bound for paraproduct-free bi-parameter singular integrals belonging to the Journ\'{e} class. See Corollary \ref{singIntSparse} for a precise statement of this result. 

The one-parameter theory indicates that we should be able to easily prove weighted estimates once we have established a sparse bound. This is still the case for the square-function sparse form estimate \eqref{squareSparse}, and we derive weighted corollaries of our main results in Sections 4 and 5.
It is also straightforward to track the dependence of the constants on the $A_{p}$ characteristic of the weight.
However, due to the addition of the square functions $S$ and some extra complications related to the strong maximal function, this approach does not give weighted estimates that are sharp in terms of the $A_{p}$ characteristic (see Section 4 for definitions). Nevertheless, our sparse bounds provide an alternative approach to proving $A_{p}$ estimates for bi-parameter martingale transforms and cancellative dyadic shifts (see \cite{HPW} for another recent method). 

\subsection{Remarks on Theorem \ref{mainThm}} \textbf{(1)} For simplicity the results and the proofs are stated for $\R \times \R$, but our methods all extend directly to the product space $\R^{n} \times \R^{m}$ once suitable modifications are made to the definition of the Haar functions. We also do not see any obstacles to carrying out the arguments below in the multi-parameter setting.

\textbf{(2)} The sparse bound \eqref{squareSparse} is true in the one-parameter setting when we are working with intervals (or cubes), but in a stronger sense. That is, \eqref{squareSparse} holds with localized square functions, so that $$|\langle Tf, g \rangle| \lesssim \sum_{I \in \mathcal{S}} |I|(S_{I}f)_{I} (S_{I}g)_{I}.$$ Here the square function $S_{I}$ only involves dyadic intervals $J$ contained in $I$ (see Theorem 15 in \cite{BB}). This localized square-function sparse bound \textit{cannot} hold in the bi-parameter setting, as observed by Lacey \cite{lacey3}. We outline the argument. Recall that $BMO_{rect}$ is the space of functions such that $$ \sup_{R_{0}} \frac{1}{|R_{0}|} \sum_{R \subset R_0} |\langle f, h_R \rangle|^{2} \leq C,$$ with the supremum taken over rectangles $R_{0}$, and $BMO_{product}$ is the space of functions such that $$ \sup_{\Omega} \frac{1}{|\Omega|} \sum_{R \subset \Omega }  |\langle f, h_R \rangle|^{2} \leq C, $$ with the supremum taken over open sets $\Omega$. There is a strict inclusion $BMO_{prod} \subset BMO_{rect}$, and in particular by using Carleson's classic counterexample we can show that for any $\epsilon > 0$ there exists $f$ such that $\|f\|_{BMO_{prod}} = 1$ but $\|f\|_{BMO_{rect}} < \epsilon$ (see Chapter 3 in \cite{MS} for the construction of this example). If we choose an open set $\Omega$ realizing the supremum for such an $f$ and assume that a localized version of \eqref{squareSparse} holds, then we would deduce \begin{align*}|\Omega| \lesssim |\langle f, f \rangle| &\lesssim \sum_{R \in \mathcal{S}} (S_{R}f)_{R} (S_{R}f)_{R}|R| \\ &\lesssim \sum_{\substack{R \in \mathcal{S} \\ R \subset \Omega }} \epsilon |R| \lesssim \epsilon |\Omega|,  \end{align*} a contradiction. 

\textbf{(3)} It is clear from our proofs of the weighted corollaries in Sections 4 and 5 that there are still significant obstacles to overcome if we wish to develop a sharp weighted theory for multi-parameter operators by using sparse domination (see, for example, the comments after the proof of Theorem \ref{martingaleTransf} and the appendix). A different notion of `sparse operator' in the multi-parameter setting may be needed, possibly one that allows us to circumvent the obstructions caused by the strong maximal function. If definitions \ref{sparseDef1} and \ref{sparseDef2} are not equivalent, it may be the case that such an operator involves collections satisfying the Carleson packing property rather than the disjoint-pieces property. 

\subsection{Notation} We write $(f)_{R}$ to denote the average $\frac{1}{|R|} \int_{R}f(y) dy$. If $w$ is a function we often write $w(R) = \int_{R} w(y) dy$. We also write $A \lesssim B$ if there is some constant $C > 0$ that only depends on the dimension or Lebesgue exponents such that $A \leq CB$. If $C$ also depends on some other parameter $\beta$, we write $A \lesssim_{\beta} B$. We also write $L^{p}(w)$ for the weighted Lebesgue space with measure $w(x)dx$. 
   
\subsection{Acknowledgments} Work leading to this paper began during the Spring 2017 semester program in Harmonic Analysis at MSRI in Berkeley, CA. The authors would like to thank the organizers. We also thank Michael Lacey and Yumeng Ou for helpful conversations, and Jos\'{e} Conde-Alonso for taking the time to read an early draft and making helpful observations.

\section{The Bi-Parameter Martingale Transform}

Fix two dyadic lattices $\mathcal{D}_1, \mathcal{D}_2$ in $\R$ and let $\mathcal{D} = \mathcal{D}_1 \times \mathcal{D}_2$ be the associated dyadic rectangles in $\R^2$. We prove the square-function sparse form bound claimed in Theorem \ref{mainThm} for the bi-parameter martingale transform \begin{equation}Tf = \label{idOp}\sum_{R \in \mathcal{D}}\epsilon_R\langle f, h_R \rangle h_R, \end{equation} where as above $\sup_{R} |\epsilon_{R}| \leq C$. The argument begins by decomposing the form $\langle Tf, g \rangle$ according to the Chang-Fefferman variant of the Calder\'{o}n-Zygmund decomposition from \cite{ChF1}. We then select a certain sparse collection of rectangles using the C\'{o}rdoba-Fefferman algorithm from \cite{CF}, and further decompose the operator in terms of these rectangles. The structure of the square function allows us to absorb the `error' terms (i.e., the rectangles not belonging to the sparse collection). 

There are a few similarities between our basic approach and the standard sparse domination scheme in the one-parameter setting. For example, the bi-parameter analogue of the Calder\'{o}n-Zygmund decomposition plays an important role in the first step. Additionally, we select the sparse collection of rectangles via a covering lemma that is equivalent to the boundedness of the strong maximal function; in the one-parameter setting, sparse cubes are typically chosen via a similar covering lemma associated to the Hardy-Littlewood maximal function.

\subsection{Proof of Theorem \ref{mainThm}} We fix two test functions $f,g$ on $\R^2$ with support in some large cube $Q_0$, and assume there are only finitely many dyadic rectangles $R$ with $\langle f, h_R \rangle$ or $\langle g, h_R \rangle$ nonzero. Let $\alpha_f = c\cdot (Sf)_{Q_0}$ and $\alpha_g = c\cdot (Sg)_{Q_0}$, where $c$ is some large constant. Define $$\Omega_0 =  \{x \in Q_0 : Sf(x) > \alpha_f \} \cup \{x \in Q_0 : Sg(x) > \alpha_g \},$$ and assume $c$ has been chosen so that $|\Omega_0| \leq \frac{1}{2}|Q_0|$. Let $\mathcal{R}_0$ be the collection of rectangles $R$ such that $|R \cap \Omega_0| < \frac{1}{2}|R|,$ and for positive integers $k$ define $$\Omega_{k} = \{x \in Q_0 : Sf(x) > 2^{k}\alpha_f \} \cup \{x \in Q_0 : Sg(x) > 2^{k}\alpha_g \}.$$ Also set $$\mathcal{F}_{k} = \{R: |R\cap \Omega_{k}| > \frac{1}{2}|R| \ \text{and} \ |R \cap \Omega_{k+1}| \leq \frac{1}{2}|R| \}.$$ We begin with the case where $\epsilon_{R} = 1$ for all $R$. We wish to estimate $$\sum_{R} \alpha_{R} = \sum_{R \in \mathcal{R}_{0}} \alpha_R + \sum_{k} \sum_{R \in \mathcal{F}_k } \alpha_{R}$$ by a sparse form (with square function averages). Observe that since $|\Omega_{k}| \rightarrow 0$ as $k \rightarrow \infty$ there are only finitely many $\mathcal{F}_{k}$ that contribute to the sum (recall that $f,g$ have only finitely many nonzero Haar coefficients). Therefore it suffices to fix a large $N$ and bound $$\sum_{R \in \mathcal{R}_{0}} \alpha_R + \sum_{k=0}^{N} \sum_{R \in \mathcal{F}_k } \alpha_{R} := I + II$$ by a sparse form, provided all constants are independent of $N$. Note that $I$ corresponds to the `good' piece in the Chang-Fefferman variant of the Calder\'{o}n-Zygmund decomposition, and $II$ corresponds to the `bad' piece. The estimate for $I$ is straightforward:

\begin{align} \nonumber \left|\sum_{R \in \mathcal{R}_0} \langle f, h_R \rangle \langle g, h_R \rangle \right| &\leq \sum_{R \in \mathcal{R}_0}\int_{R \cap \Omega_{0}^{c} }
|\langle f, h_R \rangle \langle g, h_R \rangle| \frac{\textbf{1}_{R\cap \Omega_0^{c}}(y)}{|R \cap \Omega_{0}^{c} |} \ dy  \\ \label{standardArg1} &\leq 2 \sum_{R \in \mathcal{R}_0} \int_{R \cap \Omega_{0}^{c} }
|\langle f, h_R \rangle \langle g, h_R \rangle| \frac{\textbf{1}_{R}(y)}{|R|} \ dy  \\ \nonumber &\lesssim |Q_0| (Sf)_{Q_0}(Sg)_{Q_0}.
\end{align} 

\noindent The last inequality follows from Cauchy-Schwarz and the definition of $\Omega_{0}$. To handle the remaining term $II$, we construct a sparse collection of rectangles using the C\'{o}rdoba-Fefferman selection algorithm from \cite{CF}, and decompose $II$ in terms of these rectangles. 

Fix $\beta \in (0,1)$ and begin at level $N$. Order the rectangles $\{R_i\}$ in $\mathcal{F}_N$ according to size (for example), and set  $R_{1}^{\ast} = R_1$. Proceeding inductively, choose those $R_{k}^{\ast}$ such that $$|R_{k}^{\ast} \cap \bigcup_{j < k} R_{j}^{\ast}| < \beta |R_{k}^{\ast}|,$$ and such that $R_{k}^{\ast}$ is minimal with this property relative to the initial order.  Relabel the collection $\{ R_{k}^{\ast} \}$ as $\{R_{k}^{(N)} \}$ (the rectangles in the collection at level $N$). Now suppose we have added rectangles to the collection up until level $l+1$. Let $\Lambda^{l+1}$ denote the union of all rectangles added to this point. Order the rectangles in $\mathcal{F}_{l}$ as before, and let $R_{1}^{(l)}$ be the first rectangle relative to this order such that $$|R_1^{(l)} \cap \Lambda^{l+1}| < \beta |R_{1}^{(l)}|.$$ Inductively, choose $R_{k}^{(l)}$ such that $$|R_{k}^{(l)} \cap (\bigcup_{j < k} R_{j}^{(l)} \cup \Lambda^{l+1})| < \beta |R_{k}^{(l)}|,$$ and such that $R_{k}^{(l)}$ is minimal with this property (relative to the initial order). The resulting collection $\{R^{(m)}_{j} \}_{m,j}$ is sparse in the disjoint-pieces sense, with sparse parameter $1 - \beta$. In particular, for $R = R^{(l)}_{k}$ we can choose $E_R = R \backslash (\bigcup_{j < k} R^{(l)}_{j} \cup \Lambda^{l+1}).$ By construction $|E_R| \geq (1-\beta) |R|$, and clearly $E_R \cap E_{R'} = \emptyset$ for distinct $R, R'.$

It remains to be shown that $$\bigg|\sum_{k=0}^{N}\sum_{R \in \mathcal{F}_k} \alpha_R \bigg| \lesssim \sum_{m,j}|R_{j}^{(m)}| (Sf)_{R_{j}^{(m)}}(Sg)_{R_{j}^{(m)}}.$$ Break up this sum as $$\sum_{k=0}^{N} \sum_{R = R^{(k)}_{i}} \alpha_R \ + \ \sum_{\text{rest}} \alpha_R := A + B.$$ To estimate $A$, first observe that if $\alpha_R = \alpha_{R^{(k)}_i}$ then 

\begin{align*}\alpha_R \ &= \  \int_{R_{i}^{(k)} \cap \Omega_{k+1}^{c} } \alpha_{R} \cdot  \frac{\textbf{1}_{R_{i}^{(k)} \cap \Omega_{k+1}^{c}}(y)}{|R_{i}^{(k)} \cap \Omega_{k+1}^{c} |} \ dy \\ &\leq 2\int_{R_{i}^{(k)} \cap \Omega_{k+1}^{c}} \alpha_R \frac{\textbf{1}_{R} (y)}{|R|} \ dy \\ &\leq 2 \int_{R_{i}^{(k)} \cap \Omega_{k+1}^{c}} Sf(x) Sg(x) dx. \end{align*} Now recall that $Sf \lesssim 2^{k}(Sf)_{Q_0}$ and $Sg \lesssim 2^{k}(Sg)_{Q_0}$ in $\Omega_{k+1}^{c}.$ Moreover, by construction we must have either $Sf \gtrsim 2^{k}(Sf)_{Q_0}$ or $Sg \gtrsim 2^{k}(Sg)_{Q_0}$  in more than a quarter of $R_{i}^{(k)}$. Without loss of generality suppose $Sf \gtrsim 2^{k}(Sf)_{Q_0}$. Then \begin{align}\nonumber \int_{R_{i}^{(k)} \cap \Omega_{k+1}^{c}} Sf(x) Sg(x) dx \ \nonumber&\lesssim \left(2^{k}(Sf)_{Q_0} \cdot |R_{i}^{(k)}|\right) \frac{1}{|R_{i}^{(k)}|} \int_{R_{i}^{(k)}}Sg(y)  \ dy \\ \label{stdArg}  &\lesssim |R_{i}^{(k)}| (Sf)_{R_{i}^{(k)}} (Sg)_{R_{i}^{(k)}}, \end{align} and we can ultimately conclude that \begin{equation} \label{Aest} A \lesssim \sum_{k}\sum_{i} |R_{i}^{(k)}| (Sf)_{R_{i}^{(k)}} (Sg)_{R_{i}^{(k)}}.\end{equation}

We now turn to the term $B$. Observe that a rectangle $R \in \mathcal{F}_l$ contributes to the sum $B$ if $R$ was never chosen in the C\'{o}rdoba-Fefferman selection process. It follows that for such an $R \in \mathcal{F}_{l}$ we must have $$| R \cap \bigcup_{k \geq l}\bigcup_{i} R_{i}^{(k)} \cap \Omega_{l+1}^{c}| \geq (\beta - 1/2)|R|,$$ provided we have chosen $\beta > \frac{1}{2}.$ This is because $$|R \cap \Omega_{l+1}^{c}| \geq \frac{1}{2}|R|$$ and $$| R \cap \bigcup_{k \geq l} \bigcup_{i} R_{i}^{(k)}| \geq \beta |R|. $$ We then have \begin{align}\nonumber B \ \  &\leq (\beta - 1/2)^{-1} \sum_{l} \sum_{\substack{R \in \mathcal{F}_{l} }} \alpha_R \frac{| R \cap \bigcup_{k \geq l} \bigcup_{i} R_{i}^{(k)} \cap \Omega_{l+1}^{c}| }{ |R| } \\ \label{Bargument}  &\leq (\beta - 1/2)^{-1} \sum_{l} \sum_{\substack{R \in \mathcal{F}_l }} \alpha_R \sum_{k \geq l} \sum_{i}  \frac{| R \cap R_{i}^{(k)} \cap \Omega_{k+1}^{c}| }{ |R| } \\ \nonumber &\leq (\beta - 1/2)^{-1} \sum_{k} \sum_{i} \int_{R_{i}^{(k)} \cap \Omega_{k+1}^{c}} \left( \sum_{\substack{R }} \alpha_R \frac{\textbf{1}_R}{|R|}  \right) \ dy. \end{align} We've used the fact that $\Omega_{l+1}^{c} \subset \Omega_{k+1}^{c}$ when $k \geq l$ in the second line. We can now finish the estimate by applying Cauchy-Schwarz and using the properties of the $\Omega_{k+1}^{c}$, as in \eqref{stdArg}. Hence $$B \lesssim \sum_{k}\sum_{i} |R_{i}^{(k)}| (Sf)_{R_{i}^{(k)}} (Sg)_{R_{i}^{(k)}}$$ as well, completing the proof of the case where $\epsilon_R = 1$ for all $R$.

For the general martingale transform, we simply remark that we have not used any cancellation in the above argument. Hence the argument is exactly the same if we replace $\alpha_R$ with $|\alpha_R|$, and as a consequence we may repeat the above argument with $(\sup_{R}|\epsilon_{R}|)|\alpha_R|$ in place of $\alpha_R$. The sparse bound for $\sum_{R} \epsilon_R \langle f, h_R\rangle h_R$ follows.

\section{Bi-Parameter Dyadic Shifts and Singular Integrals}

The argument from Section 2 generalizes to the case where $T$ is a cancellative bi-parameter dyadic shift if one replaces the usual square function by certain shifted variants. This ultimately leads to a type of sparse bound for paraproduct-free bi-parameter singular integrals via Martikainen's representation theorem \cite{Mart}.  

\subsection{Definitions.} We let $\mathcal{D}_1$ and $\mathcal{D}_2$ be two dyadic grids in $\R$ (not necessarily the standard grids), and let $\mathcal{D} = \mathcal{D}_1 \times \mathcal{D}_2$ be the collection of dyadic rectangles relative to these grids. If $I$ is a dyadic interval and $k \in \mathbb{N}$, we let $(I)_{k}$ denote the children of $I$ at level $k$, so that $J \in (I)_{k}$ if and only if $J \subset I$ and $|J| = 2^{-k}|I|$. We also denote the bi-parameter Haar wavelets by $h_R = h_{R_1} \otimes h_{R_2}$ for rectangles $R = R_1 \times R_2$, and use $\hat{f}(R)$ to denote the Haar coefficient of a function $f$. 

Given tuples of non-negative integers $i = (i_1, i_2)$ and $j = (j_1, j_2)$, we define the \textit{cancellative bi-parameter dyadic shift} of complexity $(i,j)$ by

 \begin{align}\label{biParamShift} T^{i,j}f(y) &= \sum_{\substack{ R_1 \in \mathcal{D}_1 \\ R_2 \in \mathcal{D}_2 }}\sum_{\substack{ P_1 \in (R_1)_{i_1} \\ P_2 \in (R_2)_{i_2} }} \sum_{\substack{ Q_1 \in (R_1)_{j_1} \\ Q_2 \in (R_2)_{j_2} }} a_{PQR} \cdot \hat{f}(P) h_{Q}(y) \end{align}
 
\noindent Here $P= P_1 \times P_2, Q= Q_1 \times Q_2$ and $R = R_1 \times R_2$ are dyadic rectangles, and $a_{PQR}$ is a constant satisfying the bound \begin{equation}\label{shiftCoefBound} |a_{PQR}| \leq \frac{\sqrt{|P_1||Q_1|} \sqrt{|P_2| |Q_2|}}{|R_1||R_2|} = 2^{-\frac{1}{2} (i_1 + j_1 + i_2 + j_2) }.\end{equation} We also define the \textit{dyadic shifted square function} adapted to the shift parameters $i,j$ by \begin{equation} \label{shiftSquare} (S^{i,j}f(y))^{2} = \sum_{\substack{ R_1 \in \mathcal{D}_1 \\ R_2 \in \mathcal{D}_2 }} \bigg(\sum_{\substack{ P_1 \in (R_1)_{i_1} \\ P_2 \in (R_2)_{i_2} }} |\hat{f}(P)| \bigg)^{2} \bigg(\sum_{\substack{ Q_1 \in (R_1)_{j_1} \\ Q_2 \in (R_2)_{j_2} }} \frac{1_{Q_1}}{|Q_1|} \otimes \frac{1_{Q_2}}{|Q_2|}(y) \bigg). \end{equation} This clearly depends on the choice of $\mathcal{D}$, but we omit this dependence from the notation since our bounds will be independent of $\mathcal{D}$. Also note that this definition is \textit{not} symmetric in $i,j$. The same is true for the bi-parameter shift of complexity $(i,j)$. (The definition \eqref{shiftSquare} is taken from the paper \cite{HPW} by Holmes, Petermichl, and Wick).

\subsection{The Sparse Bound} We will now adapt the argument from Section 2 to prove the following sparse bound. \begin{thm}\label{shiftSparse} Let $\mathcal{D}$ be an arbitrary system of dyadic rectangles. Let $T^{i,j}$ be the cancellative shift defined above (using rectangles from $\mathcal{D}$),  and fix test functions $f,g$. Then there exists a sparse collection $\mathcal{S}$ of $\mathcal{D}$-dyadic rectangles such that $$|\langle T^{i,j}f, g \rangle| \lesssim  2^{-(i_1 + i_2 + j_1 + j_2 )} \sum_{R \in \mathcal{S}} |R|(S^{i,j}f)_{R}(S^{j,i}g)_{R}.$$ The collection $\mathcal{S}$ depends on $f,g$ and $(i,j)$, but the implicit constant does not. \end{thm} 

\noindent Note that the order of $i,j$ is switched in the term containing $g$. From \cite{HPW} we know that $\|S^{i,j}f\|_{L^{p}(w)} \lesssim c_w 2^{\frac{1}{2}(i_1 + i_2 + j_1 + j_2 )} \|f\|_{L^{p}(w)}$, so we need the factor in front of the sparse form for applications to weighted estimates. In the last section we will show that in the case $p = 2$ we can at least take $c_w = [w]_{A^{2}}^{5}$. 

The proof of the theorem is similar to what we have seen above. Begin by assuming that $f,g$ are supported in some cube $L$. Let $\alpha_f = c\cdot (S^{i,j}f)_{L}$ and $\alpha_g = c\cdot (S^{j,i}g)_{L}$, where $c$ is some large constant. Define $$\Omega_0 =  \{x \in L : S^{i,j}f(x) > \alpha_f \} \cup \{x \in Q_0 : S^{j,i}g(x) > \alpha_g \}.$$ Notice that $$|\Omega_0| \leq \frac{1}{\alpha_f} \int_{L} S^{i,j}f + \frac{1}{\alpha_g} \int_{L} S^{j,i}g \leq \frac{2}{c} |L|,$$ so we can assume $c$ has been chosen independent of $(i,j)$ such that $|\Omega_0| \leq \frac{1}{2}|L|$. For positive integers $k$, also define $$\Omega_{k} = \{x \in L : S^{i,j}f(x) > 2^{k}\alpha_f \} \cup \{x \in L : S^{j,i}g(x) > 2^{k}\alpha_g \},$$ and let $\mathcal{R}_0$ be the collection of rectangles $R$ such that $|R \cap \Omega_0| < \frac{1}{2}|R|$. Finally, for $k \geq 0$ define $$\mathcal{F}_{k} = \{R: |R\cap \Omega_{k}| > \frac{1}{2}|R| \ \text{and} \ |R \cap \Omega_{k+1}| \leq \frac{1}{2}|R| \}.$$

We first estimate the `good' part of our form corresponding to the rectangles in $\mathcal{R}_0$. To further simplify the notation, if $R = R_1 \times R_2$ is a dyadic rectangle and $P= P_1 \times P_2$ is a dyadic rectangle contained in $R$, we write $P \in R_{\vec{i}}$ to mean $P_1 \in (R_1)_{i_1}$ and $P_2 \in (R_2)_{i_2}.$  

\begin{lemma} \label{indicatorEst} Let $R$ be a dyadic rectangle. Then $$\frac{1_R(y)}{|R|} =  2^{-\frac{1}{2}(i_1 + i_2 + j_1 + j_2 )} \frac{ \left( \sum_{P \in R_{\vec{i}}} 1_{P_1} \otimes 1_{P_2}(y) \right)^{1/2} \left(\sum_{Q \in R_{\vec{j}}} 1_{Q_1} \otimes 1_{Q_2}(y)\right)^{1/2} }{(|P_1||P_2|)^{1/2}(|Q_1||Q_2|)^{1/2}}.$$\end{lemma} \begin{proof} This is a simple consequence of the fact that $R$ is a disjoint union of all rectangles $P$ such that $P \in R_{\vec{i}}$, and similarly $R$ is a disjoint union of all rectangles $Q$ such that $Q \in R_{\vec{j}}$. Since $|R|^{1/2} =2^{\frac{1}{2} (i_1 + i_2)} (|P_1||P_2|)^{1/2}$ and $|R|^{1/2} = 2^{\frac{1}{2} (j_1 + j_2)}  (|Q_1||Q_2|)^{1/2}$ the identity follows. \end{proof}

\noindent Now let $$\langle T^{i,j}f, g\rangle_{\text{good}} = \sum_{R \in \mathcal{R}_{0}} \sum_{\substack{P \in R_{\vec{i}} \\ Q\in R_{\vec{j}} }} a_{PQR}\hat{f}(P) \hat{g}(Q).$$ Arguing as in the beginning of the proof of Theorem \ref{mainThm} we find that \begin{align*} |\langle T^{i,j}f, g \rangle_{\text{good}}| &\leq  2^{-\frac{1}{2}(i_1 + i_2 + j_1 + j_2 )} \sum_{R \in \mathcal{R}_{0}} \sum_{\substack{P \in R_{\vec{i}} \\ Q\in R_{\vec{j}} }} |\hat{f}(P) \hat{g}(Q)| \\ &=  2^{-\frac{1}{2}(i_1 + i_2 + j_1 + j_2 )}\sum_{R \in \mathcal{R}_0}\int_{R \cap \Omega_{0}^{c}}\sum_{\substack{P \in R_{\vec{i}} \\ Q\in R_{\vec{j}} }} |\hat{f}(P) \hat{g}(Q)|\frac{1_{R \cap \Omega_{0}^{c} } (y)}{|R \cap \Omega_{0}^{c}|} dy \\ &\lesssim 2^{-\frac{1}{2}(i_1 + i_2 + j_1 + j_2 )}\int_{\Omega_{0}^{c}}\sum_{R \in \mathcal{R}_0}\sum_{P \in R_{\vec{i}} } |\hat{f}(P)| \sum_{Q\in R_{\vec{j}} } |\hat{g}(Q)|\frac{1_{R} (y)}{|R|} dy \\ &\lesssim 2^{-(i_1 + i_2 + j_1 + j_2 )}\int_{\Omega_{0}^{c}} S^{i,j}f(y) \cdot S^{j,i}g(y) dy. \end{align*} To get to the last line, we applied Lemma \ref{indicatorEst} and then used Cauchy-Schwarz and the definition of the shifted square functions. But the integral is over $\Omega_{0}^{c}$, so we can conclude that $$|\langle T^{i,j}f, g \rangle_{\text{good}}| \lesssim  2^{-(i_1 + i_2 + j_1 + j_2 )} |L|(S^{i,j}f)_{L}(S^{j,i}g)_{L}.$$

It remains to estimate $$\langle T^{i,j}f, g \rangle_{\text{bad}} = \langle T^{i,j}f, g \rangle - \langle T^{i,j}f, g \rangle_{\text{good}}.$$ As in Section 2, this is where the sparse collection enters into the picture. We will assume that $\hat{f}(R)$ and $\hat{g}(R)$ are nonzero for only finitely many $R$, and let $N$ denote the largest integer such that $\hat{f}(R)$ and $\hat{g}(R)$ are nonzero for some $R \in \mathcal{F}_N$. All bounds will be independent of $N$, so density arguments will allow us to extend the results to more general $f,g$. We construct the sparse collection using the C\'{o}rdoba-Fefferman selection algorithm as before. The only change is that our exceptional sets $\Omega_{k}$ now depend on the shifted square function $S^{i,j}$, but otherwise the construction proceeds in exactly the same way as in the case of the martingale transform. We omit the details since the argument would be a copy of what appears in Section 2. Let $\{R_{n}^{(k)} \}$ denote the resulting collection, with $R^{(k)}_{n} \in \mathcal{F}_{k}$. We can break up the `bad' part of the form as $$\langle T^{i,j}f, g \rangle_{\text{bad}}  = \sum_{R = R^{(k)}_{n}} \sum_{\substack{P \in R_{\vec{i}} \\ Q\in R_{\vec{j}} }} a_{PQR}\hat{f}(P) \hat{g}(Q) + \sum_{\text{rest}}a_{PQR}\hat{f}(P) \hat{g}(Q) := A + B. $$ 

\textbf{Estimating A.} Fix $R = R^{(k)}_{n}$ and observe that since $R \in \mathcal{F}_k$ we have \begin{align*}\bigg|\sum_{\substack{P \in R_{\vec{i}} \\ Q\in R_{\vec{j}} }} a_{PQR}\hat{f}(P) \hat{g}(Q) \bigg| &\leq 2^{-\frac{1}{2}(i_1 + i_2 + j_1 + j_2)} \int_{R \cap \Omega_{k+1}^{c}} \sum_{\substack{P \in R_{\vec{i}} \\ Q\in R_{\vec{j}} }} |\hat{f}(P) \hat{g}(Q)|  \frac{1_{R \cap \Omega_{k+1}^{c}}(y)}{|R \cap \Omega_{k+1}^{c}|} dy \\ &\lesssim 2^{-\frac{1}{2}(i_1 + i_2 + j_1 + j_2)}\int_{R \cap \Omega_{k+1}^{c}} \sum_{P \in R_{\vec{i}} } |\hat{f}(P)| \sum_{Q\in R_{\vec{j}} } |\hat{g}(Q)| \frac{1_{R}(y)}{|R|} dy \\ &\lesssim 2^{-(i_1 + i_2 + j_1 + j_2)}\int_{R \cap \Omega_{k+1}^{c}} S^{i,j}f(y)S^{j,i}g(y) dy, \end{align*} applying Lemma \ref{indicatorEst} as above to get to the last line. Now argue as in \eqref{stdArg} and sum over all $R^{(k)}_{n}$ to get $$|A| \lesssim 2^{-(i_1 + i_2 + j_1 + j_2)}\sum_{n,k}|R^{(k)}_{n}| (S^{i,j}f)_{R^{(k)}_{n}} (S^{j,i}g)_{R^{(k)}_{n}}.$$

\textbf{Estimating B.} The term $B$ involves a sum of $\sum_{\substack{P \in R_{\vec{i}} \\ Q\in R_{\vec{j}} }}a_{PQR}\hat{f}(P) \hat{g}(Q)$ over all $R$ not chosen in the C\'{o}rdoba-Fefferman selection process. For any such $R \in \mathcal{F}_{l}$ we must have $$| R \cap \bigcup_{k \geq l} \bigcup_{n} R_{n}^{(k)} \cap \Omega_{l+1}^{c}| \geq (\beta - 1/2)|R|$$ as in Section 2, provided we have chosen $\beta > \frac{1}{2}.$ 

The proof now proceeds as in \eqref{Bargument}. One repeats the argument given in \eqref{Bargument} and makes modifications similar to what we've seen the the proof of $A$ in order to insert the shifted operators $S^{i,j}f, S^{j,i}g$. It follows that $$|B| \lesssim_{\beta} 2^{-(i_1+i_2+j_1+j_2)} \sum_{n,k}|R^{(k)}_{n}| (S^{i,j}f)_{R^{(k)}_{n}} (S^{j,i}g)_{R^{(k)}_{n}},$$ completing the proof of Theorem \ref{shiftSparse}. 

\subsection{Bi-Parameter Singular Integrals}

We can now use Martikainen's representation theorem \cite{Mart} to show that if $T$ is a \textit{paraproduct-free} bi-parameter singular integral belonging to the Journ\'{e} class, then $T$ can be estimated by an average of sparse forms of the type appearing in Theorem \ref{shiftSparse}. Loosely speaking, $T$ is a \textit{bi-parameter Journ\'{e} operator} on $\R\times \R$ if it has a kernel $K(x_1, x_2, y_1, y_2)$ on $\R^{2} \times \R^{2}$ that satisfies analogues of the Calder\'{o}n-Zygmund kernel conditions in each variable separately, along with mixed H\"{o}lder and size conditions involving the two parameters. We send the reader to Martikainen's paper \cite{Mart} for the precise definition of a Journ\'{e} operator $T$. We say that such a $T$ is \textit{paraproduct-free} if the bi-parameter version $T(1) = T^{\ast}(1) = 0$ holds, meaning there are only cancellative shifts in the dyadic representation from \cite{Mart}. 

We briefly recall one more definition. Let $\mathcal{D}_0$ denote the standard dyadic intervals in $\R$, and for every $\eta = (\eta_j)_{j \in \Z} \in \{0,1\}^{\Z}$ define the shifted grid $$\mathcal{D}_{\eta} := \{I + \eta : I \in \mathcal{D}_0 \}, \ \ \ \text{  with  } I + \eta := I + \sum_{2^{-k} < |I|}2^{-k}\eta_k.$$ We assign $\{0,1\}^{\Z}$ the natural Bernoulli(1/2) product measure. This gives us a probability measure on the space of shifted grids, and hence a probability measure on the space of shifted dyadic rectangles $\mathcal{D}_{\eta} \times \mathcal{D}_{\eta'}$ (see \cite{Hy} or \cite{Mart} for more properties of these random grids). 

\begin{cor}\label{singIntSparse} Let $T$ be a paraproduct-free bi-parameter Journ\'{e} singular integral on $\R^{2}$ and suppose $f,g$ are test functions with finitely many (bi-parameter) Haar coefficients. For each pair of tuples of non-negative integers $i,j$ set $\tau_{i,j} = 2^{-(i_1+i_2+j_1+j_2)}.$ Also let $S_{\omega}^{i,j}$ be the shifted square function \eqref{shiftSquare} defined with respect to the random dyadic system $\mathcal{D}_{\omega} = \mathcal{D}_{\omega_1} \times \mathcal{D}_{\omega_2}$. Then there exists $\delta >0$ and sparse collections of rectangles $\Lambda_{i,j}$ (depending on $f,g$) such that $$|\langle Tf, g \rangle| \lesssim \E_{\omega_1} \E_{\omega_2} \sum_{i,j \geq 0} 2^{-\max(i_1,j_1)\delta /2} 2^{-\max(i_2,j_2)\delta /2} \tau_{i,j}  \sum_{R \in \Lambda_{i,j}} |R|(S_{\omega}^{i,j}f)_{R} (S_{\omega}^{j,i}g)_{R}.$$ 
\end{cor}

\begin{proof} Given a system of shifted dyadic rectangles $\mathcal{D}_{\omega} = \mathcal{D}_{\omega_1} \times \mathcal{D}_{\omega_2}$, we let $T^{i,j}_{\omega}$ denote the shift operator \eqref{biParamShift} defined with respect to rectangles from $\mathcal{D}_{\omega}$. Martikainen proved that if $T$ is a paraproduct-free operator in the Journ\'{e} class then $$ \langle Tf, g \rangle = \E_{\omega_1} \E_{\omega_2} \sum_{i,j \geq 0} 2^{-\max(i_1,j_1)\delta /2} 2^{-\max(i_2,j_2)\delta /2} \langle T_{\omega}^{i,j}f, g \rangle.$$ The claimed result now follows by applying Theorem \ref{shiftSparse} to each form $\langle T_{\omega}^{i,j}f, g \rangle$ (recall that Theorem \ref{shiftSparse} applies for any dyadic system $\mathcal{D}_{\omega_1} \times \mathcal{D}_{\omega_2}$). 
\end{proof}

It should be possible to extend the result of Corollary \ref{singIntSparse} to arbitrary Journ\'{e} operators $T$, although the weighted estimates that follow would be far from optimal. We briefly outline one approach. It would be sufficient to prove analogues of Theorem \ref{shiftSparse} for the various paraproducts that show up in the dyadic representation of $T$. To this end, one can work with the mixed operators $SM$ and $MS$, where $S$ and $M$ are one-parameter square and maximal functions in different directions. By combining methods from \cite{HPW} or \cite{MPTT} with our sparse domination scheme, it should be possible to prove bounds of the type \begin{equation}\label{paraproduct} |\langle \Pi f, g \rangle| \lesssim \sum_{R \in \mathcal{S}} |R|(SMf)_{R}(SMg)_{R},\end{equation} where $\Pi$ is a bi-parameter paraproduct. One may also have to work with shifted variants of the mixed operators (see \cite{HPW}), and prove analogues of \eqref{paraproduct} involving these operators.

\section{Weighted Estimates for Bi-Parameter Martingale Transform}  

Let $w(x_1,x_2)$ be a positive, locally integrable weight on $\R\times \R$. Recall that $w$ is a two-parameter $A_p(\R \times \R)$ weight for $1 < p < \infty$ if and only if \begin{equation} \label{weight} [w]_{A_p (\R\times \R)}:= \sup_{R} \left(\frac{1}{|R|}\int_{R} w(x_1,x_2) \ dx \right)\left(\frac{1}{|R|}\int_{R} w(x_1,x_2)^{1 - p'} \ dx \right)^{p-1} < \infty.\end{equation} By the Lebesgue differentiation theorem this condition is equivalent to  $w(\cdot, x_2) \in A_p(\R)$ uniformly in $x_2$ and $w(x_1, \cdot) \in A_{p}(\R)$ uniformly in $x_1$, and in fact $$[w]_{A_p(\R \times \R)} \backsimeq \max( \| [w(\cdot, x_2)]_{A_{p}(\R)} \|_{L^{\infty}_{x_2}}, \ \| [w(x_1, \cdot)]_{A_{p}(\R)} \|_{L^{\infty}_{x_1}}).$$ Write $[w]_{A_p} = [w]_{A_p (\R \times \R)}.$ As we noted in the introduction, in the one-parameter setting sparse bounds lead to weighted estimates that are sharp in terms of the $A_{p}$ characteristic. Here we use our sparse bound \eqref{squareSparse} to derive $A_{p}$ estimates in terms of the bi-parameter characteristic. Unfortunately, the square-function sparse bound we have proved does not seem to imply estimates that are sharp. By using known methods, for example the arguments in \cite{HPW}, one can prove $$\|Tf \|_{L^{p}(w)} \lesssim [w]_{A_{p}(\R \times \R)}^{8} \|f\|_{L^{p}(w)}$$ for a Journ\'{e}-type operator \cite{Ou}, whereas our methods yield a power that is much worse. On the other hand, our method of proof simplifies the somewhat technical arguments that currently exist in the literature (see, for example, the remark at the end of Section 5).   

\begin{lemma} \label{iteration} Let $S$ be the bi-parameter square function with respect to some fixed dyadic grid $\mathcal{D}$, and let $M$ be the strong maximal function. Then if $w \in A_2$, $$\|Sf\|_{L^{2}(w)} \lesssim [w]_{A_2}^{2} \|f\|_{L^{2}(w)}$$ and $$\|Mf\|_{L^{2}(w)} \lesssim [w]_{A_2}^{2} \|f\|_{L^{2}(w)}$$ for all $f \in L^{2}(w)$. \end{lemma}

\begin{proof}  Recall that the dyadic one-parameter square function satisfies the weighted estimate $$\|S_{1}(f)\|_{L^{2}(w)} \lesssim [w]_{A_2(\R)}\|f \|_{L^{2}(w)}$$ and the Hardy-Littlewood maximal operator satisfies the estimate $$\|M_{1}(f)\|_{L^{2}(w)} \lesssim [w]_{A_2(\R)}\|f \|_{L^{2}(w)}.$$  Both of the claimed estimates follow by iterating the one-parameter results, using the pointwise bound $Mf \leq M_1 (M_2 f)$ for the strong maximal function (here $M_1$ is the Hardy-Littlewood operator in the direction $x_1$, and $M_2$ is the Hardy-Littlewood operator in the direction $x_2$).   \end{proof}

\begin{thm}\label{martingaleTransf} Let $\epsilon_R$ be a uniformly bounded sequence indexed over dyadic rectangles with $\sup_{R} |\epsilon_R| \leq C_{\epsilon}$, and let $Tf$ be the following bi-parameter martingale transform: $$Tf(x) = \sum_{R} \epsilon_R \langle f, h_R \rangle h_R (x).$$ Then for all $w \in A_2 = A_2(\R \times \R)$ and $f \in L^{2}(w)$ we have $$\|Tf\|_{L^{2}(w)} \lesssim C_{\epsilon}[w]_{A_2}^{8}\|f\|_{L^{2}(w)}.$$ \end{thm}

\begin{proof} The estimate follows from sparse domination. Let $\sigma = w^{-1}$. By duality it is enough to show that for all $g \in L^{2}(\sigma)$ we have $$|\langle Tf, g \rangle| \lesssim C_{\epsilon} [w]_{A_2}^{8}\|f\|_{L^{2}(w)}\|g\|_{L^{2}(\sigma)}.$$ We know from above that there is a sparse collection of rectangles $\mathcal{S}$ so that $$|\langle Tf, g\rangle| \lesssim C_{\epsilon}\sum_{R \in \mathcal{S}} |R|(Sf)_R (Sg)_R.$$ 
	
	\noindent We now a repeat a version of the standard argument from the one-parameter theory, using the strong maximal function in place of the Hardy-Littlewood maximal function. We have \begin{align*} \sum_{R \in \mathcal{S}} |R|(Sf)_R (Sg)_R &\lesssim \sum_{R \in \mathcal{S}}|E_R|  (\inf_{x\in R} M(Sf)(x) ) (\inf_{x\in R} M (Sg)(x)) \\ &\lesssim \sum_{R \in \mathcal{S}} \int_{E_R} M(Sf)(x) M(Sg)(x) w^{1/2}(x) \sigma^{1/2}(x) dx \\ &\lesssim \int_{\R^2} M(Sf)(x) M(Sg)(x) w^{1/2}(x) \sigma^{1/2}(x) dx \\ &\lesssim \|M(Sf)\|_{L^{2}(w)}\|M(Sg)\|_{L^{2}(\sigma)} \\ &\lesssim [w]_{A_2}^{4} \|Sf\|_{L^{2}(w)}\|Sg\|_{L^{2}(\sigma)} \\ &\lesssim [w]_{A_2}^{8} \|f\|_{L^{2}(w)}\|g\|_{L^{2}(\sigma)}, \end{align*} as desired.  
	
\end{proof}

\noindent By passing through a square function, it is not too hard to show that the bi-parameter martingale transform satisfies the $A_{2}$ bound $$\|Tf\|_{L^{2}(w)} \lesssim [w]_{A_{2}}^{3} \|f\|_{L^{2}(w)},$$ hence the constants in Theorem \ref{martingaleTransf} are far from optimal. We do not know if the power of 8 appearing in Theorem \ref{martingaleTransf} can be pushed down further using our methods. In the one-parameter setting, the usual argument that produces the sharp $A_2$ bound invokes the weighted maximal operator \begin{equation}\label{weightedMax} M_1^{\mu}f(x) = \sup_{x \in I} \frac{1}{\mu(I)} \int_{I} |f(y)| \mu(y) dy,\end{equation} which is bounded on $L^{2}(\mu)$ for any positive function $\mu$, with norm independent of $\mu$ (this follows from the Besicovitch covering lemma or martingale theory, see \cite{S} for example). However, the bi-parameter analogue of \eqref{weightedMax} is in general \textit{not} bounded on $L^{2}(\mu)$, due to the more complicated geometry. R. Fefferman proved in \cite{Fef} that $\mu \in A_{\infty}(\R\times \R)$ is sufficient for the strong weighted maximal function $M^{\mu}$ to be bounded on $L^{2}(\mu)$, but the sharp dependence of the operator norm on $[\mu]_{A_{\infty}}$ is unclear from his argument. It is somewhat surprising that if we trace the dependence in his argument and use some recent sharp results related to $A_{\infty}$ (\cite{HP}, \cite{HPR}), we uncover a dependence that is \textit{exponential} in the $A_{\infty}$ characteristic. Recall that $w$ is in the bi-parameter weight class $A_{\infty}(\R \times \R)$ if $w$ is in the one-parameter class $A_{\infty}(\R)$ uniformly in each variable. 

\begin{prop} \label{strongWeighted} Suppose $w \in A_{p}(\R \times \R)$ and let $M^{w}$ be the two-dimensional weighted strong maximal function $$ M^{w}(f)(y) = \sup_{y \in R} \frac{1}{w(R)}\int_{R} |f(x)| \ w(x) dx.$$ Then for all $1 < p  < \infty$ we have $\|M^{w}\|_{L^{p}(w) \rightarrow L^{p, \infty}(w) } \lesssim_{p} [w]_{A_p}e^{c[w]_{A_{\infty}}}.$ \end{prop}

\noindent We prove this proposition in the appendix. 

The sparse bounds from Section 3 also allow us to derive weighted estimates for dyadic shifts and paraproduct-free Journ\'{e} operators. The argument is almost the same as the proof of Theorem \ref{martingaleTransf}, but in this case we have to work with weighted estimates for the shifted square functions $S^{i,j}$. We know from \cite{HPW} that if $w \in A_{p}(\R \times \R)$ there is some $c_{w} > 0$ such that $$\|S^{i,j}f\|_{L^{p}(w)} \leq 2^{(i+j)/2}c_w \|f\|_{L^{p}(w)}$$ for all $f \in L^{p}(w),$ but we would like to track the dependence of $c_w$ on $[w]_{A_p}$. This is the content of the next section.

\section{Weighted Estimates for the Shifted Square Function}

It was proved in \cite{HLW} that the one-parameter shifted square function $$S_{1}^{i,j}f(x)^{2} = \sum_{R \in \mathcal{D}} \bigg( \sum_{P \in (R)_i} |\hat{f}(P)| \bigg)^{2} \sum_{Q \in(R)_j}\frac{1_Q(x)}{|Q|}$$ satisfies the weighted estimate $\|S_{1}^{i,j}f\|_{L^{2}(w)} \leq  2^{(i+j)/2}C_{w}\|f\|_{L^{2}(w)}$ for $w \in A_{2}$. In particular, the argument in \cite{HLW} gives $C_{w} \leq [w]_{A_{2}}^{2}.$ In this section we prove a type of sparse bound for $S^{i,j}$ that allows us to show $C_w \leq [w]_{A_2}^{1/2}[w]_{A_{\infty}}^{1/2}$. An iteration argument then shows that the bi-parameter analogue of $S_{1}^{i,j}$ satisfies a weighted bound with constant $c_{w} \lesssim [w]_{A_2}^{4}[w]_{A_{\infty}}$. The method of proof is an adaptation of the scalar case of the argument by Hyt\"{o}nen, Petermichl, and Volberg in \cite{HPV}.  
 
 \subsection{Preliminary Results}
 
Fix an arbitrary (one-paramter) dyadic lattice $\mathcal{D}$. 
 
\begin{lemma}\label{sublinear} Suppose $f_k$ is a sequence of functions such that $S_{1}^{i,j}(f_k)$ is defined for each $k$. Then $S_1^{i,j}(\sum_{k}f_k) \leq \sum_{k}S^{i,j}(f_k)$. \end{lemma}
 
\begin{proof} Fix an arbitrary $x \in \R$. The lemma is a simple consequence of Minkowski's inequality for the weighted space $\ell^{2}(1_{R}(x)/|R|)$, where $\|\{\alpha_R\}\|_{\ell^{2}(1_{R}/|R|)}^{2} = \sum_{R} \alpha_{R}^{2} \frac{1_R}{|R|}$. Let $F_{k, (R)_i} = \sum_{P \in (R)_{i}}|\hat{f_k}(P)|. $ Then \begin{align*} S_1^{i,j}(\sum_{k} f_{k}) &= 2^{j/2}\| \sum_{k} F_{k, (R)_i} \|_{\ell^{2}(\frac{1_R}{|R|})} \\ &\leq 2^{j/2} \sum_{k}  \| F_{k, (R)_i} \|_{\ell^{2}(\frac{1_R}{|R|})} \\ &= 2^{j/2} \sum_{k} \left( \sum_{R} (F_{k, (R)_{i}})^{2} \frac{1_R}{|R|} \right)^{1/2} = \sum_{k}S_1^{i,j}(f_k).    \end{align*}
 \end{proof}

\begin{prop}\label{weak11} The operator $S_1^{i,j}$ maps $L^{1}(\R)$ into $L^{1,\infty}(\R)$ with $\|S_1^{i,j}\|_{L^{1} \rightarrow L^{1,\infty}} \lesssim 2^{(i+j)/2}.$ 
\end{prop}

\begin{proof} The argument is a variation of the standard approach via the Calder\'{o}n-Zygmund decomposition. Fix $f \in L^{1}(\R)$ and $\lambda, \alpha > 0$. Choose maximal dyadic intervals $J$ such that $\frac{1}{|J|}\int_{J}|f| > \alpha\lambda$ and let $\Omega$ denote the union of such intervals. Then $f = g + b$, with $g = f1_{\Omega^{c}} + \sum_{J} (f)_{J}1_{J}$ and $b = \sum_{J}(f - (f)_{J})1_{J}$. Moreover $\|g\|_{L^{\infty}} \lesssim \alpha\lambda$ and $\|b_{J}\|_{L^{1}} \lesssim \alpha\lambda|J|.$

By Lemma \ref{sublinear} we have $$|\{S_1^{i,j}f > \lambda\}| \leq |\{S_1^{i,j}g > \lambda/2 \}| + |\{S_1^{i,j}b > \lambda/2 \}|.$$ Using the $L^{2}$-boundedness of $S_1^{i,j}$ we can immediately conclude that $$ |\{S_1^{i,j}g > \lambda/2 \}| \lesssim \lambda^{-2}2^{i+j}\|g\|^{2}_{L^{2}} \lesssim 2^{i+j}\frac{\alpha}{\lambda}\|f\|_{L^{1}}.$$ Let $E = \bigcup_{J} 5J$ and note $|E| \lesssim \alpha^{-1}\lambda^{-1}\|f\|_{L^{1}}$. We also claim that \begin{equation}\label{czStep}|\{ x \in E^{c} : S^{i,j}b(x) > \lambda/2 \} | \lesssim \frac{\alpha}{\lambda}2^{i+j}\|f\|_{L^{1}}.\end{equation}  We will show that \begin{equation} \label{L2cz} \int_{E^{c}} S_1^{i,j}(b)(x)^{2} dx \leq 2^{i+j}\alpha\lambda \sum_{J}\|b_J\|_{L^{1}}, \end{equation} which will be enough to prove \eqref{czStep} since it will imply \begin{align*}| \{ x \in E^{c} : S_1^{i,j}b(x) > \lambda/2 \}| &\lesssim \lambda^{-2} \int_{E^{c}} S_1^{i,j}(b)(x)^{2} dx \\&\lesssim \lambda^{-2} (2^{i+j}\alpha\lambda \sum_{J}\|b_{J}\|_{L^1}) \lesssim 2^{i+j} \frac{\alpha}{\lambda}\|f\|_{L^{1}}. \end{align*}

\noindent To prove \eqref{L2cz}, we apply Lemma \ref{sublinear} to get \begin{align} \nonumber\int_{E^{c}} S_1^{i,j}(b)(x)^{2} dx &\leq 2^{j} \int_{E^{c}} \sum_{J} \sum_{R \in \mathcal{D}} \left(\sum_{P \in (R)_i} |\widehat{b_J}(P)| \right)^{2} \frac{1_R (x)}{|R|}dx \\ \label{L2czEst1} &= 2^{j} \sum_{J}\int_{E^{c}}  \sum_{\substack{|R| > |J| \\ R \supset J }} \left(\sum_{P \in (R)_i} |\widehat{b_J}(P)| \right)^{2} \frac{1_R (x)}{|R|}dx.  \end{align} Notice that only the intervals $R$ with $R \supset J$ contribute to the sum, since $E^{c} = (\bigcup_{J} 5J)^{c}$ and if $J \cap R = \emptyset$ and $P \in (R)_i$ then $\widehat{b_J}(P) = 0$. Now \begin{align*} \left(\sum_{P \in (R)_i} |\widehat{b_J}(P)| \right)^{2} &\leq \left( \int_{\R} |b_{J}(x)| \cdot \sum_{P \in (R_i)} |h_P(x)|  dx \right)^{2} \\ &= \frac{2^{i}}{|R|} \left( \int_{\R} |b_{J}(x)| \sum_{P \in (R)_i} (|P|^{1/2}|h_{P}(x)|) dx \right)^{2},\end{align*} and since $\sum_{P \in (R)_i} |P|^{1/2}|h_{P}(x)|$ is bounded independent of $i$ (due to the disjointness of $P \in (R)_i$) it follows that \begin{equation} \label{haarEst} \left(\sum_{P \in (R)_i} |\widehat{b_J}(P)| \right)^{2} \lesssim \frac{2^{i}}{|R|} \left(\int_{\R} |b_J(x)| dx\right)^{2} \lesssim \|b_{J}\|_{L^{1}} 2^{i} \frac{\alpha\lambda|J|}{|R|}. \end{equation} Inserting \eqref{haarEst} into \eqref{L2czEst1} yields \begin{align*} \int_{E^{c}} S_1^{i,j}(b)(x)^{2} dx  &\lesssim 2^{i+j} \alpha\lambda \sum_{J} \|b_{J}\|_{L^{1}} \sum_{\substack{|R| > |J| \\ R \supset J}} \frac{|J|}{|R|} \\ &\lesssim 2^{i+j} \alpha\lambda \sum_{J} \|b_{J}\|_{L^{1}},  \end{align*} proving \eqref{L2cz}. In summary, we have shown $$|\{S_1^{i,j}f > \lambda \}| \lesssim \big(2^{i+j}\frac{\alpha}{\lambda} + \frac{1}{\alpha\lambda} + 2^{i+j}\frac{\alpha}{\lambda}\big)\|f\|_{L^{1}}$$ for arbitrary $\alpha > 0$. Setting $\alpha = 2^{-(i+j)/2}$ yields $$|\{S_1^{i,j}f > \lambda \}| \lesssim 2^{(i+j)/2} \frac{1}{\lambda} \|f\|_{L^{1}}, $$ completing the proof. 
\end{proof}

\subsection{The Sparse Bound }

The weak bound for $S_1^{i,j}$ allows us to mimic the sparse domination scheme from \cite{HPV}. A simple computation shows that \begin{equation}\label{shiftedNorm}\|S_1^{i,j}f\|^{2}_{L^{2}(w)} = 2^{j} \sum_{R \in \mathcal{D}} \bigg(\sum_{P \in (R)_{i}} |\hat{f}(P)|\bigg)^{2} (w)_{R}.\end{equation} We will estimate the term on the right by the norm of a sparse operator. We assume there are only finitely many Haar coefficients of $f$, so there is some large interval $J$ that contains every interval contributing to the sum in \eqref{shiftedNorm}. Fix large constants $C_1, C_2 > 0$ to be determined below, and begin by choosing maximal dyadic intervals $L$ such that either \begin{equation} \label{stopping1} \sum_{R\supset L} \bigg( \sum_{P \in (R)_i} |\hat{f}(P)| \bigg)^{2}\frac{2^{j}}{|R|}  > 2^{i+j}C_1 (|f|)^{2}_{J}\end{equation} or \begin{equation}\label{stopping2} (w)_{L} > C_2 (w)_{J}. \end{equation} Let $\mathcal{S}_1'$ denote the collection of maximal intervals from $\eqref{stopping1}$, and let $\mathcal{S}_{1}''$ denote the collection of maximal intervals from \eqref{stopping2}. The initial collections are $\mathcal{S}_0 = \{ J\}$ and $\mathcal{S}_1 = \mathcal{S}_1' \cup \mathcal{S}_1''$.  We have 

\begin{align*} 2^{j}\sum_{R} \bigg( \sum_{P \in (R)_i}|\hat{f}(P)|\bigg)^{2}(w)_{R} &= 2^{j}\sum_{\substack{R \text{ s.t. } \forall L \in \mathcal{S}_1 \\ R\nsubseteq L}} \bigg( \sum_{P \in (R)_i}|\hat{f}(P)|\bigg)^{2}(w)_{R} \  \\ & \ \ \ \ \ \ \ \ \ \ + 2^{j}\sum_{\substack{R \text{ s.t. } \exists L \in \mathcal{S}_1 \\ R \subset L}} \bigg( \sum_{P \in (R)_i}|\hat{f}(P)|\bigg)^{2}(w)_{R} \\ &:= A + B. \end{align*}

\noindent To estimate $A$ we use the stopping conditions \eqref{stopping1} and \eqref{stopping2}: \begin{align*} A &= 2^{j} \sum_{\substack{R \text{ s.t. } \forall L \in \mathcal{S}_1 \\ R\nsubseteq L}} \bigg( \sum_{P \in (R)_i}|\hat{f}(P)|\bigg)^{2}(w)_{R}\\  &\leq 2^{j} C_2\sum_{\substack{R \text{ s.t. } \forall L \in \mathcal{S}_1 \\ R\nsubseteq L}}  \bigg( \sum_{P \in (R)_i}|\hat{f}(P)|\bigg)^{2}(w)_{J} \\ &\leq  C_2\sum_{\substack{R \text{ s.t. } \forall L \in \mathcal{S}_1 \\ R\nsubseteq L}}  \bigg( \sum_{P \in (R)_i}|\hat{f}(P)| \bigg)^{2}2^{j}\frac{|J|}{|R|}(w)_{J}  \\ &\leq 2^{i+j}C_1C_2  |J| (|f|)_{J}^{2}(w)_{J}.
\end{align*} 

\noindent The term $B$ may be handled by recursion, by decomposing it as a sum of operators of type \eqref{shiftedNorm} localized to each $L$. The same selection process is used at each iteration, with the same constants $C_1, C_{2}$.  It remains to check that the stopping intervals actually form a sparse collection if we choose $C_1$ and $C_2$ correctly, and also that we can choose $C_1, C_2$ independent of $i,j$.  

We claim that all of the intervals $L$ chosen in \eqref{stopping1} are contained in $\{ (S_1^{i,j}f)^{2} >  2^{i+j}C_1 (|f|)^{2}_{J} \}.$ In fact, $$(S_1^{i,j}f(x))^{2}1_{L}(x) \geq 2^{j} \sum_{R \supset L} \bigg(\sum_{P \in (R)_{i}} |\hat{f}(P)|\bigg)^{2}\frac{1_{R}(x)1_{L}(x)}{|R|} > 2^{i+j}C_1(|f|)_{J}^{2}\cdot 1_{L}(x) $$ by selection, which proves the claim. It follows from Proposition \ref{weak11} that we can choose $C_{1} \sim 1$ such that $\sum_{L \in \mathcal{S}_{1}'} |L| \leq \frac{1}{4}|J|$. For the intervals chosen in \eqref{stopping2}, we directly estimate the following sum: 

$$C_2 \sum_{L \in \mathcal{S}_1''} |L| \leq  \sum_{L \in \mathcal{S}_1''} (w)_{J}^{-1} \int_{L}w(x) dx \leq |J|, $$ using the disjointness of the $L \in \mathcal{S}_{1}''$. Hence if $C_2 = 4$ then $\sum_{L \in \mathcal{S}_1''} |L| \leq \frac{1}{4}|J|$, and as a consequence $\sum_{L \in \mathcal{S}_1} |L| \leq \frac{1}{2} |J|$. Moreover, we can choose $C_1$ and $C_2$ independently of $i,j$. The same choice of $C_1, C_{2}$ at each iteration guarantees that the collection is sparse. We have proved the following:

\begin{prop}  Suppose $f$ has finitely many Haar coefficients and fix non-negative integers $i,j$. Then there exists a sparse collection $\mathcal{S}$ of dyadic intervals such that \begin{equation}\label{oneParamShiftSparse}\|S_1^{i,j}f\|^{2}_{L^{2}(w)} \lesssim 2^{i+j} \sum_{J \in \mathcal{S}}|J| (w)_{J}(|f|)_{J}^{2}.\end{equation}  The implicit constant is independent of $i,j$ and $f,w$. 
\end{prop}

\begin{cor} \label{squareShiftSharp}  Suppose $f \in L^{2}(w)$ with $w\in A_2$ and fix non-negative integers $i,j$. Then $$\|S_1^{i,j}f\|_{L^{2}(w)} \lesssim [w]_{A_2}^{1/2}[w]_{A_{\infty}}^{1/2}2^{(i+j)/2}\|f\|_{L^{2}(w)}.$$
\end{cor}

\begin{proof} We use a special case of the general argument outlined in \cite{HPV}. Note that we have estimated $\|S_1^{i,j}f\|_{L^{2}(w)}$ by the (scalar version of the) same sparse object appearing in that paper.
	
By plugging $w^{-1/2}f$ into the sparse bound $\eqref{oneParamShiftSparse}$, we get $$\|S_1^{i,j} (w^{-1/2}f)\|^{2}_{L^{2}(w)} \lesssim 2^{i+j} \sum_{J \in \mathcal{S}}|J| (w)_{J}(|f|w^{-1/2})_{J}^{2}.$$ Hence it will be enough to show that the sum on the right above is no more than $C[w]_{A_2}[w]_{A_{\infty}}\|f\|^{2}_{L^{2}(\R)}.$ Let $\delta = \frac{1}{c[w]_{A_{\infty}}}$ and $r = 2(1 + \delta)$. By H\"{o}lder's inequality $$(|f| w^{-1/2})_{J}^{2} \leq (|f|^{r'})_{J}^{2/r'} (w^{-(1+\delta)} )^{\frac{1}{1+ \delta}}_{J},$$ hence the reverse H\"{o}lder inequality yields $$(|f| w^{-1/2})_{J}^{2} \lesssim (|f|^{r'})_{J}^{2/r'} (w^{-1})_{J}$$  (see \cite{HPR}). It follows that \begin{align*} \sum_{J \in \mathcal{S}}|J| (w)_{J}(|f|w^{-1/2})_{J}^{2} &\lesssim \sum_{J \in \mathcal{S}}|J| (|f|^{r'})_{J}^{2/r'} (w^{-1})_{J} (w)_{J} \\ &\lesssim [w]_{A_2} \sum_{J \in \mathcal{S}} |J| (|f|^{r'})_{J}^{2/r'} \\ &\lesssim [w]_{A_2} \int_{\R} M(|f|^{r'})_{J}^{2/r'}(x) dx, \end{align*} using the sparsity of the collection to get the integral over $\R$ in the last line. Note that $r' < 2$, hence $$\|M^{r'} (f) \|_{L^{2}(\R)}^{2} \lesssim ((2/r')')^{2/r'}\|f\|^{2}_{L^{2}(\R)}$$ with $(2/r')'$ the dual exponent to $2/r'$. Using the definition of $r$ we see that  $((2/r')')^{2/r'} \lesssim [w]_{A_{\infty}}$, and as a consequence we can conclude that $$ \sum_{J \in \mathcal{S}}|J| (w)_{J}(|f|w^{-1/2})_{J}^{2} \lesssim [w]_{A_{2}}[w]_{A_{\infty}} \|f\|^{2}_{L^{2}(\R)}$$ as desired.  

\end{proof}

\textbf{Remark.} It is now straightforward to prove weighted estimates for bi-parameter dyadic shifts and the type of Journ\'{e} operators considered in Corollary \ref{singIntSparse}, although these estimates are far from sharp. In particular, by using Corollary \ref{squareShiftSharp} and the iteration argument from Section 3 in \cite{HPW}, one can show that the bi-parameter shifted square function $S^{i,j}$ satisfies the $A_{2}$ bound $\|S^{i,j}f\|_{L^{2}(w)} \lesssim [w]^{4}_{A_{2}}[w]_{A_{\infty}}$ (the extra powers come from passing through a martingale transform). Then argue as in the proof of Theorem \ref{martingaleTransf}, with the $S^{i,j}$ replacing the simpler square functions $S$. Note that the dependence on $[w]_{A_2}$ would be improved if we could prove a sharper weighted estimate for $S^{i,j}$.

\section*{Appendix: The Weighted Strong Maximal Function}

Here we prove Proposition \ref{strongWeighted}. As in \cite{Fef}, the idea is to bootstrap the boundedness of the maximal function in dimension one with the help of the $A_{\infty}$ property of the weight. We follow R. Fefferman's argument and also use some recent `weighted Solyanik estimates' due to P. Hagelstein and I. Parissis \cite{HP}, which rely on the sharp reverse H\"{o}lder estimates in \cite{HPR}. 
	
Fix $1 < p < \infty$. By the covering-lemma argument in \cite{CF}, it is enough to show that if $R_1, R_2,...$ is a sequence of rectangles with sides parallel to the axes, then there is a subcollection $\{\widetilde{R}_{j}\}$ of $\{ R_j \}$ such that \begin{equation} \label{covLem1} \int_{\bigcup \widetilde{R}_j} w(x) dx \geq C_1 \int_{\bigcup R_j} w(x) dx \end{equation} and \begin{equation} \label{covLem2} \| \sum_{j} 1_{\widetilde{R}_j} \|_{L^{p'}(w)} \leq C_2 \bigg(\int_{\bigcup R_j} w(x) dx\bigg)^{1/p'},  \end{equation} where $p'$ is the dual exponent to $p$. In this case one has $$w(\{M^{w}f > \alpha \})^{1/p} \leq C_1^{-1} C_2 \frac{\|f\|_{L^{p}(w)}}{\alpha}, $$ so that $\|M^{w}\|_{L^{p}(w) \rightarrow L^{p,\infty}(w)} \leq C_1^{-1}C_2.$ 
	
By monotone convergence we may assume the initial sequence $\{R_j\}$ is finite. We choose the subcollection $\{ \widetilde{R}_{j} \}$ using the C\'{o}rdoba-Fefferman selection algorithm from \cite{CF}. Assume the rectangles $\{R_j\}$ have been ordered with decreasing sidelengths in the $x_2$ direction, and take $\widetilde{R}_1  = R_1$. Proceeding inductively, let $\widetilde{R}_j$ be the first $R_k$ occurring after $\widetilde{R}_{j-1}$ so that $$ |R_k \cap \bigcup_{l < k} \widetilde{R}^{\ast}_l | < (1-\epsilon) |R_k|,$$ where $\epsilon \in (0, e^{-c[w]_{A_\infty}} )$. Then arguing as in the proof of Cor. 5.3 from \cite{HP}, we conclude that \begin{align*}w(\bigcup_{j} R_{j}) &\leq (1 + c\epsilon^{(c[w]_{A_{\infty}})^{-1} } ) w(\bigcup_{k} \widetilde{R}_k) \\ &\leq A\cdot w(\bigcup_{k} \widetilde{R}_k), \end{align*} with $A$ independent of $w$ (we've used the assumed upper bound on $\epsilon$). Therefore we can take $C_1$ independent of $w$ in \eqref{covLem1}. 
	
Now take a point $\bar{x} = (\alpha, \beta)$ inside a rectangle $R_k$ which does not occur among the $\widetilde{R}_j$. Let $\{T_i\}$ denote the intervals obtained by slicing the two-dimensional rectangles $\{R_i\}$ with a line perpendicular to the $x_2$ axis at height given by $\beta$ (the $x_2$-coordinate of $\bar{x}$). Given any rectangle $R = I \times J$, write $R^{\ast} = I \times 3J$. We claim that for each such $R_i \sim T_i \times J_i$ we must have \begin{equation} \label{sliceId1} |T_i \cap \bigcup \widetilde{T}^{\ast}_j | \geq (1-\epsilon)|T_i|,\end{equation} with $\widetilde{T}^{\ast}_j$ the slices corresponding to $\widetilde{R}^{\ast}_j$. This follows from the assumption about decreasing sidelengths. In fact, we may assume all $\widetilde{T}_j^{\ast}$ appearing in the union correspond to $\widetilde{R}_j$ that intersect $R_i$ and were chosen before $R_i$ relative to the initial order (the full union is only larger). By the selection criterion $$|R_i \cap (\bigcup \widetilde{R}^{\ast}_j)| \geq (1-\epsilon)|R_i| = (1-\epsilon)|T_i||J_i|.$$ But the sidelengths of the $\widetilde{R}_j$ parallel to the $x_2$ axis are longer than $J_i$, and in particular their three-fold dilates contain $J_i$. It follows that $R_i \cap (\bigcup \widetilde{R}^{\ast}_j) = (T_i \cap \bigcup \widetilde{T}_{j}^{\ast} ) \times J_i. $ This implies \eqref{sliceId1}, since we must have $\frac{|R_i \cap (\bigcup \widetilde{R}^{\ast}_j)|}{|J_i|} = |T_i \cap \bigcup \widetilde{T}_{j}^{\ast} |. $
	
Next, observe that if $E_j = \widetilde{T}_j - \bigcup_{l < j} \widetilde{T}_l^{\ast}$ then by arguing as above (and using the selection criterion) we see that $|\widetilde{T}_j \cap \bigcup_{l < j} \widetilde{T}^{\ast}_l| \leq (1-\epsilon)|\widetilde{T}_j|$, and therefore $|E_j|/|\widetilde{T}_j| > \epsilon.$ Since $w \in A_{p}$ we can conclude that $$\epsilon^{p} \leq \left(\frac{|E_j|}{|\widetilde{T}_j|}\right)^{p} \leq [w]_{A_p} \frac{w(E_j)}{w(\widetilde{T}_j)} $$ uniformly in the free variable. Hence for any $f \in L^{p}(wdx_1)$ with $\|f\|_{p} \leq 1$ we have 
	
\begin{align*}\int \sum_{j} 1_{\widetilde{T}_j}(x_1)f(x_1) w(x) dx_1 &\leq  \epsilon^{-p} [w]_{A_p} \int_{E_j} w(x) dx_1 \cdot \frac{1}{\int_{ \widetilde{T}_j } w(x)dx_1 } \int_{\widetilde{T}_j } f(x_1) w(x) dx_1 \\ &\leq   \epsilon^{-p} [w]_{A_p} \int_{\bigcup \widetilde{T}_j} M^{w}_1 (f)(x_1) w(x)dx_1 \\ &\leq  \epsilon^{-p} [w]_{A_p}\|M^{w}_1 f \|_{L^{p}(w dx_1)} \left(\int_{\bigcup \widetilde{T}_j}w(x) dx_1 \right)^{1/p'} \\ & \lesssim  \epsilon^{-p} [w]_{A_p}\left(\int_{\bigcup \widetilde{T}_j}w(x) dx_1 \right)^{1/p'}, \end{align*} using the fact that the weighted one-dimensional maximal operator is bounded independent of $w$. We also used the disjointness of the sets $E_{j}$ to sum. After integrating in $x_2$ it follows that we can take $C_2 =  \epsilon^{-p} [w]_{A_p}$ in \eqref{covLem2}.  
	
Combining the above results yields $$\|M^w\|_{L^{p}(w) \rightarrow L^{p,\infty}(w)} \lesssim  \epsilon^{-p} [w]_{A_p}$$ for any $\epsilon \in (0, e^{-c[w]_{A_\infty}})$. Hence $\|M^w\|_{L^{p}(w) \rightarrow L^{p,\infty}(w)} \lesssim [w]_{A_p}e^{c[w]_{A_{\infty}}}$ as claimed.

We do not know if there is an alternative approach the the boundedness of $M^{w}$ that yields a smaller dependence on $[w]_{A_{\infty}}$. Note, however, that Fefferman's covering lemma is equivalent to the boundedness of $M^{w}$ on $L^{p}(w)$, up to constants (see \cite{CF}).

\Addresses

\end{document}